\documentclass[10pt]{scrartcl} 

\usepackage[utf8]{inputenc}
\usepackage[T1]{fontenc}
\usepackage{lmodern}

\usepackage{amsfonts,amssymb,amsmath,amsthm}
\usepackage{graphicx}  
\usepackage[dvipsnames]{xcolor}
\usepackage{natbib}
\usepackage{authblk}
\usepackage[]{enumitem}
\usepackage{tabularx}
\usepackage{subfigure}
\usepackage{mathtools}
\usepackage{url}
\usepackage{algorithm} 
\usepackage{algpseudocode} 
\usepackage{comment}
\usepackage{pgfplotstable}
\usepackage{csvsimple}
\usepackage{array,multirow,makecell}
\usepackage{upgreek}

\newtheorem{theorem}{Theorem}[section]

\newtheorem{corollary}[theorem]{Corollary}
\newtheorem{definition}[theorem]{Definition}

\newtheorem{example}[theorem]{Example}

\newcommand{\eps}{\ensuremath{\epsilon}}

\newcommand{\NP}{{\sffamily\textbf{NP}}}

\usepackage[color=blue!30, textsize=scriptsize]{todonotes}
\usepackage{changes}
\usepackage{booktabs}
\definechangesauthor[color=blue!50!black]{MS}
\definechangesauthor[color=green!50!black]{LL}
\definecolor{mediumcandyapplered}{rgb}{0.89, 0.02, 0.17}
\definecolor{grayishblue}{rgb}{0.93, 0.93, 0.95}
\definecolor{navyblue}{rgb}{0.0, 0.0, 0.5}

\DeclarePairedDelimiter\abs{\lvert}{\rvert}

\newcommand{\X}{\ensuremath{X}}

\newcommand{\YN}{\ensuremath{Y_{N}}}

\newcommand{\R}{\mathbb{R}}
\newcommand{\Z}{\mathbb{Z}}

\newcolumntype{C}[1]{>{\centering\arraybackslash }b{#1}}

\usepackage{tikz, pgfplots}
\usetikzlibrary{arrows, shapes.misc, chains, patterns}
\usetikzlibrary{decorations.pathreplacing} 
\usetikzlibrary {decorations.pathmorphing} 
\usepackage{xifthen}

\tikzset{
  candidat/.style={rectangle, inner sep=0pt, minimum size=0.1cm, draw=gray, fill=gray},
  nds/.style={circle, inner sep=0pt, minimum size=0.12cm, draw=black, fill=black},
  ndns/.style={rectangle, inner sep=0pt, minimum size=0.12cm, draw=black, fill=black},
  test/.style={circle, inner sep=0pt, minimum size=0.12cm, draw=black, fill=black},
}
\pgfplotsset{compat=1.8}

\begin{document}

\title{A Multi-objective Perspective on the Cable-Trench Problem}

\author{Lara L\"ohken}
\author{Michael Stiglmayr}
\affil{University of Wuppertal, School of Mathematics and Natural Sciences, IMACM, Gaußstr.~20, 42119 Wuppertal, Germany\\
\{loehken,stiglmayr\}@uni-wuppertal.de
}

\date{}
\maketitle

\begin{abstract}
  The cable-trench problem is defined as a linear combination of the shortest path and the minimum spanning tree problem. In particular, the goal is to find a spanning tree that simultaneously minimizes its total length and the total path length from a pre-defined root to all other vertices. Both, the minimum spanning tree and the shortest path problem are known to be  efficiently solvable. However, a linear combination of these two objectives results in a highly complex problem.

  In this article we consider bi-objective cable-trench problem which separating the two cost functions. We show that in general the bi-objective formulation has additional compromise solutions compared to the cable-trench problem in its original formulation. In order to determine the set of non-dominated points and efficient solutions, we use $\eps$-constraint scalarizations in combination with a problem specific cutting plane. Moreover, we present numerical results on different types of graphs analyzing the impact of density and cost structure on the cardinality of the non-dominated set and the solution time.
 
 \bigskip
 \emph{Keywords: multi-objective optimization, combinatorial optimization, minimum spanning tree, shortest paths, network flow problem}
\end{abstract}

\section{Problem Formulation} \label{sec:ProblemFormulation}
\subsection{Cable-Trench Problem}

Both, the shortest path problem and the minimum spanning tree problem represent widely studied problems in graph theory that are known to be efficiently solvable since the mid-20th century. Although integer programming problems generally form a complex class of problems in mathematical programming, combinatorial problems such as the shortest path problem and the minimum spanning tree problem benefit from their underlying structure in the respective solution process. A comprehensive introduction to integer and combinatorial programming can be found in, e.g.,  \cite{Wols1998}, \cite{NemWols1988} and \cite{korte18}. For graph theory and network optimization in particular see, e.g., \cite{deo74}, \cite{GrossYellen05} and \cite{ahuja93}. There exist various approaches to efficiently determine a minimum spanning tree for a connected, edge-weighted undirected graph. Classic algorithms are, for example, the algorithm of Prim \citep{Prim_1957} or the algorithm of Kruskal \citep{Kruskal_1956}.  The problem of finding shortest paths from a source vertex to all other vertices in a connected, undirected graph with non-negative edge weights, also known as single-source shortest path problem (SSSP), can be efficiently solved by, for example, the algorithm of Dijkstra \citep{Dijkstra_1959}. However, combining the minimum spannning tree problem and the single-source shortest path problem results in a computationally demanding problem as we will discuss later in this work in more detail.

There exist several publications investigating the interrelation of minimum spanning tree and shortest path problem. A linear algorithm for an edge cost sensitivity analysis for both problems is presented in \cite{Booth_1994}. In \cite{Vasko}, it is pointed out that in general the minimum spanning tree does not correspond to a tree of shortest paths w.r.t.\ any root node in a given graph. \cite{Khuller_1995} discuss the trade-off between the total weight of a spanning tree and the distance in the tree between a given root vertex to any other vertex. \cite{Eppstein} shows that the minimum spanning tree that minimizes the path length between a particular set of vertices is \NP-complete. Finally, \cite{Vasko} introduced the \emph{cable-trench problem} (CTP) which combines shortest path problem and minimum spanning tree problem in a single-objective optimization problem.

\begin{definition}[CTP, cf.~\citealt{Vasko}] \label{VaskoDefCTP}
	For a connected graph $G = (V,E)$ with specified root vertex  $v_0 \in V$, let $c(e) \geq 0$ be the cost of each $e \in E$. Furthermore, let $\tau \geq 0$ and $\gamma \geq 0$ be fixed positive parameters. The solution of the cable-trench problem (CTP) is the spanning tree $T$  that minimizes $ \tau \, c_{\tau}(T) + \gamma \, c_{\gamma}(T)$ where $c_{\tau}(T)=\sum_{e\in T} c(e)$ is the total cost of the spanning tree $T$ and $ c_{\gamma}(T)=\sum_{v\in V}\sum_{e\in P(v)} c(e)$ is the total path cost in $T$ from $v_0$ to all other vertices of $V$. Thereby, \(P(v)\) denotes the (unique) path from \(v_0\) to \(v\) in \(T\).
\end{definition}

The name-giving application of the cable-trench problem is the task of creating a campus network where each building on the campus is connected to a central server $v_0$ via its own dedicated cable. The goal is to realize this at minimum cost. Digging trenches between all buildings, in which the cables can be laid is a classical application of the minimum spanning tree problem. Whereas the design of a network such that the total cable length is minimized can be formulated as a single source shortest path problem. Note that these two optimization models do not necessarily lead to the same optimal solutions.

It can be easily seen from Definition \ref{VaskoDefCTP} that for $\gamma > 0$ and $ \tau = 0$ an optimal solution of the CTP is a shortest path tree (SPT) of $G$ with root vertex $v_0$. Conversely, for $\tau >0$ and $\gamma = 0$, an optimal solution of the CTP is a minimum spanning tree of $G$. As a result, the solutions for the two marginal cases can be determined in polynomial time whereas the CTP will be shown to be highly computationally demanding in general.

Following on from this, solutions of the CTP can be categorized with respect to the ratio $R = \tau / \gamma$, including the two limiting cases $R \rightarrow 0$ for $\tau = 0$ and $R \rightarrow \infty$ for $\gamma = 0$. \cite{Vasko} show that, if minimum spanning tree and shortest path tree differ from each other, then the solutions of the CTP, more precisely optimal solutions for any non-negative values of $\tau$ and $\gamma$, can be described by a sequence of spanning trees as the ratio $R$ increases. 

To solve the problem for all weight ratios $R$, \cite{Vasko} propose a heuristic, based on a neighborhood search procedure, that tends to generate either optimal or nearby optimal solutions very quickly. Furthermore, \cite{Vasko} show how a heuristic, not necessarily optimal solution combined with optimal solutions for a few problem-specific ratio values $R$ can be used to generate tight lower bounds for the CTP and moreover optimal solutions for all weight ratios $R$.

Motivated by a the healthcare application of vasculature reconstruction from (micro) computer tomographic scans, \cite{Vasko2015}\ define a more generalized version of the CTP in 2015, the generalized cable-trench problem (GCTP). See \cite{Jiang}\ for a general description of vasculature reconstruction. \cite{Vasko2015}\ introduce a setting of the CTP in which each edge of the graph is assigned two in general different cost coefficients for the two components of the objective function, namely the spanning tree objective and the shortest path objective from the root node. In terms of the application of creating a campus network, the cost of laying a cable and the cost of digging a trench between a pair of nodes generally differs. More formally, the edge costs can be written as $c_{e} = (t_{e}, s_{e})$ and a solution of the GCTP is a spanning tree $T$ that minimizes $ \tau \, t_{\tau}(T) + \gamma \, s_{\gamma}(T) $, where analogously $t_{\tau}(T)$ matches the total spanning tree cost and $s_{\gamma}(T)$ the total path cost in $T$. Thus, a general motivation for the GCTP lies in the idea that the path cost of an edge can differ from the edge cost for a spanning tree. Note that, if both costs of each edge are proportional to each other, then the GCTP reduces to the CTP. \cite{Vasko2015} develop two highly efficient metaheuristics that are capable of finding near-optimal solutions to GCTPs in quadratic time in the number of vertices of the graph. With regard to this application, large scale instances with up to $25000$ vertices are of particular interest.

Several other extensions of the cable-trench problem are introduced and investigated in the literature. \cite{Marianov_2012} address the $p$-cable-trench problem ($p$CTP), for which not only one but $p$ source vertices exist. They present an integer programming formulation of the problem as well as Lagrangean relaxation-based heuristics in order to solve large scale instances that require high computation time if solved exactly. In order to improve the solution process even more, \cite{Schwarze2016} propose a metaheuristic for the $p$CTP. Further extensions of the $p$CTP are discussed in \cite{Marianov2015} and \cite{Calik_2017}, in which a variant of the $p$CTP with covering is formulated and considered. \cite{Calik_2017} furthermore includes capacities and present an algorithmic framework based on Benders decomposition in order to compute solutions for the capacitated $p$CTP with covering which can also be applied to most other variants of the CTP. By introducing the capacitated $p$CTP with facility cost, \cite{Schwarze_2020} also takes into account capacity in terms of customer limitations associated with a facility and, in addition, the cost of opening a facility. 

There exist several other extensions of the CTP. The multi-commodity cable-trench problem introduced by \cite{Schwarze2015TheMC} sets up a formulation of the problem including different cable types from different operators. The goal is for operators to reduce their individual costs by using the same trenches. \cite{Landquist2018} define the generalized steiner cable-trench problem in which not every vertex of the graph has to be connected to $v_0$ and which offers the possibility of in general different edge costs. Like the GCTP, this problem is applied to a healthcare application as well, namely to error correction in vascular image analysis.

While we will discuss the complexitiy of the CTP and associated results by \cite{Eppstein} and \cite{Vasko} in more detail later in this work, findings regarding the complexity on the approximability of the problem can be found in \cite{Benedito2021,BENEDITO2023}. More precisely, \citeauthor{Benedito2021}\ show that the approximation with factor of at most 1.000475 is \NP-hard. In \cite{davis2023classifying} tractable special cases of the cable-trench problem are considered and classified according to the underlying graph structure.

In this work, we will focus on the cable-trench problem in its original version and provide a new perspective on its objective function in terms of multi-objective optimization. This approach can be directly applied to the GCTP and furthermore extended to other versions of the CTP. Therefore, we start in Section \ref{Section Bi-Objective CTP} by giving a brief introduction into multi-objective opti,mization in general, followed by a formulation for the bi-objective cable-trench problem. In Section \ref{sec:TheoreticalResults}, we show some theoretical results on the bi-objective cable-trench problem and its computational effort. Moreover, we discuss differences between the single and bi-objective perspective. After that, we present a solution strategy and adaptations in order to speed up the solution process, followed by numerical results.

\subsection{Bi-Objective Cable-Trench Problem} \label{Section Bi-Objective CTP}

The cable-trench problem as defined in \cite{Vasko} minimizes two cost functions, namely $c_{\tau}(T)$ and $c_{\gamma}(T)$, respectively $t_{\tau}(T)$ and $s_{\gamma}(T)$ for its generalized version. Thereby, the objective function consists of a linear combination with weights $\tau$ and $\gamma$. The relative importance of both cost functions is then represented by the ratio $R = \tau / \gamma$.

Instead of using a weighted linear combination of both cost functions, we consider both cost functions as two seperate objectives that are minimized seperately. In doing so, we change the perspective from single to multi-objective. Therefore, we first give a brief introduction to multi-objective optimization in general. 

In many practical applications, not only one but several, in general conflicting objectives are relevant to evaluate the quality of a feasible solution. In other words, there are several objective functions to be optimized. Multi-criteria decision making and multi-objective optimization, in particular, offer the possibility to optimize w.r.t.\  two or more objectives simultaneously. A comprehensive introduction to multi-objective optimization can be found, e.g., in \cite{Ehr2005}. In the following, we will focus on multi-objective combinatorial optimization (MOCO). In combinatorial optimization, feasible solutions correspond to a subset of elements of a finite ground set, e.g., edges of a graph. For an introduction to and survey of multi-objective combinatorial optimization see also \cite{EhrGand2000}. Since any combinatorial problem can be formulated as an integer program, we start with a definition of a general multi-objective integer linear program.

\begin{definition}
	A \emph{multi-objective integer linear program} can be written as 
	\begin{align} \label{MOILP}
		\min \big\{ \left( f_{1}(x), \dots, f_{p}(x) \right) = Cx : x \in X \big\}
		\tag{MOILP}
	\end{align}
	where $p \geqq 2$ and $C \in \R^{p \times n}$. $X$ denotes the set of feasible solutions of the problem and is defined by $X : = \{ x \in \Z^{n}: Ax \leqq b, x \geqq 0\} $ with $A \in \R^{m \times n}$, $b \in \R^{m}$ and $C \in \R^{p \times n}$. The feasible outcome set $Y$ is defined by $Y:= \{ Cx: x \in X \}$
\end{definition}

When considering multiple objective functions, we cannot expect to find a solution which optimizes all objective functions simultaneously. Instead of considering the preferences of a potential decision maker, we follow an a-posteriori approach using the Pareto concept of optimality and determinig the whole set of so-called non-dominated points. Pareto optimality relies on componentwise orders in $\R^{p}$. We write $y^{1} \leqq y^{2}$, if $y_{k}^{1} \leqq y_{k}^{2}$ for $k = 1 \dots p$, $y^{1} \leqslant y^{2}$, if $y^{1} \leqq y^{2}$ and $y^{1} \neq y^{2}$, and $y^{1} < y^{2}$, if $y_{k}^{1} < y_{k}^{2}$ for $k = 1 \dots p$. Accordingly, the set $\R^{p}_{\geqq} := \{x \in \R^{p}: x \geqq 0 \}$ and analogously $\R^{p}_{\geq}$ and $\R^{p}_{>}$ are defined.

\begin{definition}
	A feasible solution $x^{\ast} \in X$ is called an \emph{efficient (weakly efficient)} solution of \eqref{MOILP}, if there does not exist a feasible solution $x \in X$ such that $f(x) \leqslant f(x^{\ast})$ $( f(x) < f(x^{\ast}))$. If $x^{\ast}$ is (weakly) efficient, then $f(x^{\ast})$ is called \emph{(weakly) non-dominated}. If $f(x) \leqslant f(x')$ holds for $x, x' \in X$, then $x$ \emph{dominates} the point $x'$ and $f(x)$ dominates the point $f(x')$. The set $X_{E} \subseteq X$ denotes the set of all efficient solutions and the set $Y_{N} = \{ y \in \R^p : \exists \text{ } x \in \X_E : f(x) = y \}$ analogous the set of all non-dominated points, also called the non-dominated or Pareto optimal set. 
\end{definition}

Note that, the non-dominated set is componentwise bounded by the ideal point $y^{I}$ and the Nadir point $y^{N}$ that are defined by $y^{I}_{k} := \min_{y \in Y} y_{k}$ and $y^{N}_{k} := \max_{y \in Y_{N}} y_{k}$ for $k = 1 \dots p$. 

We distinguish the following classes of efficient solutions.

\begin{enumerate}
	\item \emph{Supported} efficient solutions are efficient solutions which are optimal for an associated weighted-sum scalarization 
	\begin{align} \label{weightedsumproblem} 
		\min \big\{ \lambda_{1}f_{1}(x) + \dots + \lambda_{p}f_{p}(x) :  x \in X \big\} \tag{$\text{P}_\lambda$}
	\end{align}
	
	for some weight vector $\lambda \in \R^{p}_{>}$. The corresponding images in the objective space are called supported non-dominated points. By $X_{SE}$ we denote the set of all supported efficient solutions and by $Y_{SN}$ the set of all supported non-dominated points. Note that all supported non-dominated points are located on the boundary of the convex hull of $Y$, i.e. they are non-dominated points of $(\text{conv } Y) + \R^{p}_{\geqq}$. 
	
	Depending on where they are located on $(\text{conv } Y) + \R^{p}_{\geqq}$, they can be further distinguished.
	\begin{enumerate}
		\item Supported efficient solutions $x \in X_{SE}$ for which the objective vector $f(x)$ is an extremepoint of $(\text{conv } Y) + \R^{p}_{\geqq}$ are called \emph{extreme supported efficient solutions}, $X_{SE1}$, and \emph{non-dominated extreme points}, $Y_{SN1}$, respectively.
		\item Supported efficient solutions $x \in X_{SE}$  that are located in the relative interior of an efficient face of $\text{conv } Y$ belong to the set of \emph{non-extreme supported efficient solutions}. For such a solution $x$ there exist supported extreme efficient solutions $x^{1}, \dots, x^{p}$ such that $f(x^{i}) \neq f(x^{j})$ all $i, j \in \{1, \dots, p\}$ with $i \neq j$, and $\tilde{\upalpha} \in \{ \upalpha \in \R^p_{\geq} : 0 \leq \upalpha_{i} < 1 \text{ } \forall \text{ } i = 1, \dots , p, \quad \sum_{i=1}^{p} \upalpha_{i} = 1	\} $ with $f(x) = \sum_{i = 1}^{p} \tilde{\upalpha}_{i} f(x^{i})$. Those solutions and their corresponding outcome vectors are denoted by $X_{SE2}$ and $Y_{SN2}$, respectively.
	\end{enumerate}
	
	\item \emph{Nonsupported} efficient solutions are efficient solutions that are not optimal solutions of  a weighted-sum scalarization $P_{\lambda}$ for any $\lambda \in \R^{p}_{>}$. 
	
\end{enumerate}

Multi-objective programs with $p = 2$ are called bi-objective problems (BOP). Accordingly, MO is replaced by BO in all multi-objective problem class acronyms dealing with two objectives. 
There are several types of algorithms to determine non-dominated points for multi-objective programs. We thereby aim at computing a minimal complete set, i.e., the set of all non-dominated points and, for each of them, a corresponding efficient solution.
In this sense, we present a bi-objective formulation of the cable-trench problem, in which both the total path cost and cost of the spanning tree are minimized separately.

\begin{definition}
	For a connected graph $G = (V,E)$ with specified vertex  $v_0 \in V$, let $c_e = c_{ij} \geqslant 0 $ for all $ e = [i,j] \in E$ be edge cost coefficients. The cable-trench problem is written as the bi-objective optimization problem
	\begin{align} \label{BCTP}
		\min &  \text{ } c_{\gamma}(T) \notag \\
		\min & \text{ } c_{\tau}(T) \tag{BCTP} \\
		\text{s.t.} & \text{ } T \in \mathcal{T} \notag
	\end{align}
	where $\mathcal{T}$ denotes the set of all spanning trees of G, $ c_{\gamma}(T)$ the total path cost in $T$ of all paths from $v_0$ to all other vertices of $V$ and $c_{\tau}(T)$ the total cost of the spanning tree $T$. \newline
	
	Furthermore, let $\widetilde{c}_{e}$ be edge cost coefficients given by $\widetilde{c}_{e} = (s_{e},t_{e}) = (s_{ij}, t_{ij})$ with $s_{ij}, t_{ij} \geqslant 0 $ for all $e = [i,j] \in E$. Then the generalized CTP can be formulated accordingly as the bi-objective problem
	
	\begin{align} \label{BGCTP}
		\min &  \text{ } s_{\gamma}(T) \notag \\
		\min & \text{ } t_{\tau}(T) \tag{BGCTP} \\
		\text{s.t.} & \text{ } T \in \mathcal{T} \notag
	\end{align}
	where $s_{\gamma}(T)$ denotes the total path cost, determined by edge costs $s_{e}$, and $t_{\tau}(T)$ the total cost of the spanning tree $T$, determined by the cost coefficients $t_{e}$.
\end{definition}

This bi-criteria formulation directly illustrates that the BCTP is a combination of the minimum spanning tree and the shortest path problem, as both objective functions correspond to exactly one of the two problems. Furthermore, the single-objective and the bi-objective formulation are related to each other, if we consider the CTP from a multi-objective point of view, since the formulation of \citeauthor{Vasko}\ resembles a weighted-sum scalarization of both cost functions. In the next section, we will examine this in more detail comparing both formulations.

\section{Theoretical Results} \label{sec:TheoreticalResults}
\subsection{Differences Between Single- And Multi-Objective Perspective}

Considering the (generalized) cable-trench problem from a bi-objective perspective not only offers a new interpretation of the problem itself, but also gives rise to new algorithmic approaches. Moreover, we have an idea of how both formulations are related and can be embedded in each other. Thereby, the question arises if the two formulations are equivalent in the sense that both problems lead to the same set of solutions when varying $\tau, \gamma \geqslant 0$.

Since the CTP can be interpreted as a weighted sum scalarization of the BCTP, it follows from theory of multi-objective optimization that any optimal solution of the CTP for $\tau, \gamma \geqslant 0$ is also an efficient solution of the BCTP, respectively for the GCTP and BGCTP. Namely, the respective weighted sum problem yields the set of supported non-dominated points $Y_{SN}$ in the solution space. Thus, by solving the BCTP, we can determine all solutions of the CTP as well. The question remains whether non-supported non-dominated points may exist that can not be found via a weighted sum scalarization. This would imply that there are nonsupported efficient solutions of BCTP that can not be obtained by solving the original CTP formulation.

\begin{figure}
	\begin{minipage}[b]{0.45\textwidth}
		\begin{center}
			\begin{tikzpicture}[scale=0.85,auto,>=stealth',shorten >=1pt,auto,node distance=1cm and 1cm, main node/.style={circle,draw=black,font=\footnotesize}, text=black, small node/.style={circle,draw=black,font=\footnotesize}]
				
				\node[main node] (4) at (0,0) {4};
				\node[main node] [above  = of 4] (1) {$S$};
				\node[main node]  [right = of 1] (2) {2};
				\node[main node]  [below = of 2] (3) {3};
				\node (graph) at (1,0.5) {\small $G$};
				
				\node[main node] (5) at (3.0,0) {4};
				\node[main node] [above  = of 5] (6) {$S$};
				\node[main node]  [right = of 6] (7) {2};
				\node[main node]  [below = of 7] (8) {3};
				\node (t1) at (4,.5) {\small$T_1$};
				
				\node[main node] (9) at (0,-3.0) {4};
				\node[main node] [above  = of 9] (10) {$S$};
				\node[main node]  [right = of 10] (11) {2};
				\node[main node]  [below = of 11] (12) {3};
				\node (t2) at (1,-2.5) {\small $T_2$};

				\node[main node] (13) at (3.0,-3.0) {4};
				\node[main node] [above  = of 13] (14) {$S$};
				\node[main node]  [right = of 14] (15) {2};
				\node[main node]  [below = of 15] (16) {3};
				\node (t3) at (4,-2.5) {\small $T_3$};

				\path[every node/.style={font=\footnotesize}]
				(1) edge [color = black] node [above] {5} (2)
				(1) edge [color = black] node [left] {10} (4) 
				
				(2) edge [color = black] node [right] {6} (3) 
				
				(3) edge [color = black] node [below] {4} (4)

				(6) edge [color = black] node [above] {5} (7)
				(6) edge [color = black] node [left] {10} (5) 
				
				(7) edge [color = black] node [right] {6} (8)

				(10) edge [color = black] node [above] {5} (11)
				(10) edge [color = black] node [left] {10} (9)

				(12) edge [color = black] node [below] {4} (9)

				(14) edge [color = black] node [above] {5} (15)
				
				(15) edge [color = black] node [right] {6} (16) 
				
				(16) edge [color = black] node [below] {4} (13)
				
				;
			\end{tikzpicture}  
		\end{center}
		\caption{Graph \(G\) and its spanning trees $T_1,T_2,T_3$ (see~Example~\ref{noncvx_example})}
		\label{noncv_CTP_example_graph_Fig}
	\end{minipage}
	\hfill
	\begin{minipage}[b]{0.48\textwidth}	
		\begin{center}
			\begin{tikzpicture}[scale=0.45]				
				\draw[line width=0.8,->,black!75](24,13) -- (24,24) node[above] {$f_1 = c_{\gamma}$};
				\draw[line width=0.8,->,black!75] (24,13) -- (35,13) node[above left] {$f_2 = c_{\tau}$};
				\draw[help lines,color=black!75, xstep=1cm] (25,12.9) grid (34,13.1);
				\draw[help lines,color=black!75, ystep=1cm] (23.9,14) grid (24.1,23);
				
				\foreach \x in {25,30,35}
				\node[color=black!75] at (\x,12.2) {\x};
				
				\foreach \y in {15,20}
				\node[color=black!75] at (23.2,\y) {\y};
				
				\fill [fill=grayishblue] (26,24) -- (26,21) -- (31,15) -- (35,15) decorate [decoration=random steps, segment length=3pt] {|- cycle} ; 
				\filldraw (26,21) circle (3pt) node [right] {\color{black!75} \small $f(T_1)$};
				\filldraw (29,19) circle (3pt) node [right] {\color{black!75} \small $f(T_2)$};
				\filldraw (31,15) circle (3pt) node [above right] {\color{black!75} \small $f(T_3)$};
				
				\draw[dashdotted, thick, color = mediumcandyapplered] (26,24.2) -- (26,21) -- (31,15)-- (35.2,15);
				\node[right, color = black!75] (Y+R) at (27.5,22.5) {\small $\text{conv } (Y) + \R^{2}_{\geqq}$};
			\end{tikzpicture}
		\end{center}
		\caption{Non-dominated set \(Y_N\) in outcome space (see Example~\ref{noncvx_example})}
		\label{nondomSet_noncvx_exampleFig}
	\end{minipage}
\end{figure}

\begin{example} \label{noncvx_example}
	Consider the graph given in Figure~\ref{noncv_CTP_example_graph_Fig} with root node $S$. The set of non-dominated points of the corresponding bi-objective cable-trench problem is $Y_{N} = \{ (26,21)^{\top}, (29,19)^{\top}, (31,15)^{\top}\}$, see Figure~\ref{nondomSet_noncvx_exampleFig}. Note that $(29,19)^{\top}$ is an unsupported non-dominated point, which can be recognized by the fact that it is not located on the boundary of the convex hull of $Y + \R^{2}_{\geqq}$. Recall that only points located on the convex hull of $Y + \R^{2}_{\geqq}$, namely the points in $Y_{SN}$, can be determined via a weighted sum scalarization. This shows that, in general, solving the BCTP may yield additional solutions that can not be obtained by solving the corresponding CTP for any values of the weights \(\tau,\gamma\geq 0\).
\end{example}

In fact, these additional unsupported solutions can be substantial. For this, let us take a closer look at Example \ref{noncvx_example}. Solving the CTP yields as optimal solutions the minimum spanning tree or the shortest path tree depending on the choice of $\tau, \gamma \geqslant 0$. If, alternatively, we solve the BCTP, we obtain an additional compromise solution that lies between the two extreme cases. From the point of view of a decision maker this compromise solution could be of considerable interest.

Consequently, the multi-objective view not only allows a new interpretation of the (generalized) cable-trench problem and its solutions from a theoretical point of view, but also has practical implications.

\subsection{Computational Hardness}
In this section, we will study the computational hardness of solving the CTP and GCTP. While \cite{Vasko} already explored the complexity of the CTP in its originally formulation, we will extend these investigations to the bi-objective case. We will examine two different characteristics, the complexity and intracatability of the problem. Both of which provide insights into how computationally hard the problem can be. A self-contained introduction to complexity theory and intractability is given, for example, in \cite{Garey1979} and \cite{Ehr2005}.

First, by examining the theoretical results presented by \cite{Vasko}, we analyse if the decision problem of whether efficient solutions generally exist is polynomially solvable. Subsequently, we investigate whether the number of non-dominated points is polynomially bounded.

\subsubsection{Complexity of the (Generalized) Cable-Trench Problem}
Recall the two objective functions of the bi-objective (generalized) cable-trench problem, namely the total edge costs and total path lengths from the predefined root node to all other vertices. Both refer to well known efficiently solvable single-objective optimization problems, one being the minimum spanning tree problem and the other the (single-source) shortest path problem. Nevertheless, as we will see, the bi-objective combination and the weighted sum of both problems will be computationally hard.

First, we take a closer look at the following two subproblems: finding a shortest path tree such that the total edge costs are minimized and finding a minimum spanning tree that minimizes the path cost from a starting node to all other vertices in the graph. From a bi-objective perspective, each of these subproblems refer to determining the respective lexicographic minima of both objectives, see, e.g., \cite{Ehr2005}.

\begin{theorem}[\citealt{Vasko}] \label{lexikoOptCTP1}
	The problem of finding a spanning tree that is a shortest path solution from $v_{0}$ to all other vertices in $V$ such that the total edge costs are minimized is polynomially solvable.
\end{theorem} 
\begin{proof}
	\cite{Vasko} presents the following proof by Eppstein. Determine shortest paths from $v_0$ to each vertex of $V \setminus \{v_0\}$. Find the set of all directed edges that can be part of a shortest path, i.e., an edge $(u, w)$ is in this set if 
	$c(u,w) = \sum_{e \in P_w} c(e) - \sum_{e \in P_u} c(e)$, where $P_w$ and $P_u$ denote paths from $v_0$ to $w$ and $u$, respectively.
	Find a minimum spanning arborescence of this directed graph. All the above steps can be carried out in polynomial time. 
\end{proof}

Theorem \ref{lexikoOptCTP1} states that finding the lexicographically optimal solution for the GCTP w.r.t.\ first minimizing $s_{\gamma}(T)$ and among those optimal solutions minimizing $t_{\tau}(T)$, which can be denoted as $\text{lexmin}_{T \in \mathcal{T}}(s_{\gamma}(T), t_{\tau}(T))$, can be efficiently accomplished. The same holds for the CTP with $s_{\gamma}(T) = c_{\gamma}(T)$ and $ t_{\tau}(T)= c_{\tau}(T)$. As any lexicographically optimal solution is also a (supported) efficient solution \citep{Ehr2005}, we directly see that one non-dominated point can be determined in polynomial time. Accordingly, the decision problem of whether an efficient solution exists is also polynomially solvable. Furthermore, in the special case that there exists a feasible solution for the (generalized) CTP which minimizes both objective functions simultaneously, it it can be computed in polynomial time. From now on we assume that the ideal point is infeasible. 

Following up on this, let us take a look at the lexicographic optimization problem with respect to the reversed order of the objective functions, namely the problem of finding a minimum spanning tree that minimizes the total path cost. Since our assumption holds, this will lead to different solutions than those from the problem stated in Theorem \ref{lexikoOptCTP1}. 

\begin{theorem}[\citealt{Vasko}] \label{lexikoOptCTP2} 
	The problem of finding the minimum spanning tree that minimizes the total path costs from a starting node $v_{0}$ to all other vertices $v_i \in V\setminus\{v_0\}$ is \NP-complete.
\end{theorem}

This theorem can be proven based on a result by \cite{Eppstein} which states that finding a minimum spanning tree that minimizes the path costs between a given pair of vertices is \NP-complete. This is shown by reduction form 3-SAT.

Thus, \cite{Vasko} shows that determining $\text{lexmin}_{T \in \mathcal{T}}( c_{\tau}(T),c_{\gamma}(T))$ is  \NP-complete. Again, as any lexicographically optimal solution is also a (supported) efficient point, determining at least one non-dominated point is \NP-complete. From that, we can follow two statements.

\begin{corollary}[\citealt{Vasko}]\label{CTPNPcomplete}
	Let $\gamma, \tau > 0$, then the CTP is \NP-complete.
\end{corollary}

Hence, a weighted combination of two efficiently solvable single-objective problems results in a computationally hard optimization problem. 
In \cite{dellamico96multicriteria} several weighted sums of arborescence problems are considered and their computational complexity is analyzed in comparison to their respective single-objective parts. The authors observe there as well that the combination of objectives can lead to \NP-hard problems while each of the original problems is polynomially solvable.
Note that, if $\gamma$ or $\tau$ are equal to 0, the CTP can be solved in polynomial time, since we are not looking for the lexicographic optimal solutions but any tree that is optimal either for $c_{\gamma}$ or $c_{\tau}$, which are efficiently solvable problems.

Furthermore, the following holds for the BCTP.

\begin{theorem} \label{BCTPNPcomplete}
	Assume that no feasible solution minimizes both objective functions of the \ref{BCTP} at the same time. Then, the STOP problem, which either computes another non-dominated point or returns that no such point exists, for enumerating the non-dominated set for the \ref{BCTP} is \NP-complete. 
\end{theorem}

\begin{proof}
	As the BCTP defines a bi-objective problem and no feasible solution simultaneously minimizes both objectives, there exist at least two different non-dominated points, which can be determined by the two lexicographic optimization problems $\text{lexmin}_{T \in \mathcal{T}}( c_{\gamma}(T),c_{\tau}(T))$ and $\text{lexmin}_{T \in \mathcal{T}}(c_{\tau}(T),c_{\gamma}(T))$. 
	
	Since the problem $\text{lexmin}_{T \in \mathcal{T}}(c_{\tau}(T),c_{\gamma}(T))$ is already \NP-hard (Theorem~\ref{lexikoOptCTP2}), the statement follows.
\end{proof}

Since the Cable-Trench Problem is a special case of its generalized formulation in which the edge cost coefficients for both objective functions remain the same, we can state the corrolary given below.

\begin{corollary} \label{GBCTPNPcomplete}
	The STOP problem for enumerating the non-dominated set for the \ref{BGCTP} is \NP-complete, assuming that the ideal point is infeasible.
\end{corollary}

Note that the statement from Theorem \ref{BCTPNPcomplete} also holds for problem instances where all efficient solutions are strictly efficient, which means that every non-dominated point has a uniquely defined preimage.

To summarize, although the associated single-objective problems are efficiently solvable, the same does not hold true for its bi-objective combination, in which both objective functions are minimized simultaneously, even if only two non-dominated points exist.

\subsubsection{Intractability}

Whereas in the multi-objective setting complexity is often concerned with how computationally hard it is to determine individual non-dominated points, intractability is concerned with the cardinality of the non-dominated set. More precisely, we are interested in whether the number of non-dominated points increase exponentially with the size of the instance.

\begin{theorem}
	The bi-objective Cable-Trench Problem (\ref{BCTP}) is intractable. 
	\label{Theorem_Intractable}
\end{theorem}

\begin{proof}
	To show the intractability of the Cable-Trench Problem, we construct an instance such that $\abs{\YN} $ grows exponentially in $n$. Let $V = \{v_0, \ldots, v_n\}$ be a set of vertices, where $n$ is an even number. Then, we define three different sets of edges which form a kind of windmill structure with $\frac{n}{2}$ blades, each consisting of three edges as shown in Figure~\ref{exm_intractable}.  Thereby, all nodes $v_1, \dots, v_n$ are directly connected with $v_0$ by edges $a_k$ or $d_k$ for $k = 1, 2, \dots, \frac{n}{2}$. Moreover, there are edges $b_k, k = 1, 2, \dots, \frac{n}{2}$, connecting nodes $v_{2k-1}$ and $v_{2k}$.
	
	In analogy to the intractability proof for the bi-objective shortest path problem \citep{hansen1979}, the basic idea is to choose edge costs such that: (1) there is an exponential number of efficient solutions (represented by spanning trees) and (2) the efficient solutions map to pairwise distinct non-dominated points.
	
	We will see that the edge costs in Table~\ref{Table_Edgecosts_Intractable} achieve this goal. Note that the BCTP can be solved on each windmill blade separately, as changing a solution in the $k$-th blade affects the objective functions $c_{\tau}$ and $c_{\gamma}$ only in the $k$-th digit. 
	Moreover, in each windmill blade $k$ the edge $a_k$ is part of any efficient solution, since any solution including $b_k$ and $d_k$ is dominated by the solution using $a_k$ and $b_k$. Then, in each windmill blade, there remain exactly two options to generate a spanning tree, namely by choosing either edge $b_k$ or edge $d_k$ with $k = 1, \dots, \frac{n}{2}$. In total there are $2^{\frac{n}{2}}$ efficient solutions, which map to $2^{\frac{n}{2}}$ distinct non-dominated points.
\end{proof}

\begin{table}
	\centering
	\small 
	\begin{tabular}{lll}
		\toprule
		Edge costs \\
		\midrule
		$c(e) = 3 \cdot 10^{k-1}  $, & for $e = a_{k}= [v_0, v_{2k-1}],$ 
		&$k = 1,2 \ldots, \frac{n}{2}$\\[2ex]
		$c(e) = 4 \cdot 10^ {k-1 }$, & for $e = d_{k} = [v_0, v_{2k}], $
		&$k = 1,2 \ldots, \frac{n}{2}$\\[2ex]
		$c(e) = 2 \cdot 10^ {k-1 }$, & for $e = b_{k} =  [v_{2k-1}, v_{2k}], $
		&$k = 1,2 \ldots, \frac{n}{2}$ \\
		\bottomrule
	\end{tabular}	
	\caption{Intractability example: edge costs (see Proof of Theorem~\ref{Theorem_Intractable})}
	\label{Table_Edgecosts_Intractable}
\end{table}

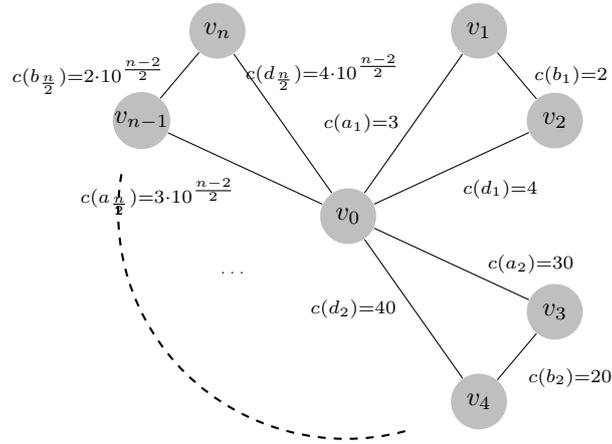
\begin{figure}
	\centering
	\begin{tikzpicture}[shorten >=0.001pt,-, scale=0.75]
		\tikzstyle{vertex}=[circle,fill=black!25,minimum size=21pt,inner sep=0pt]
		
		\foreach \name/\beta in {1/55, 2/25, 3/-25, 4/-55, n-1/-205, n/-235}
		\node[vertex] (G-\name) at (\beta:4) {$v_{\name}$};
		\node[vertex] (G-0) at (0,0) {$v_{0}$};
		\draw [thick, dashed] (1,-3.8) arc (-75:-190:4) ;

		\tikzstyle{vertex}=[ {circle, minimum size=10pt,inner sep=0pt}]
		\node[vertex] (G-points) at (-2,-1) {\tiny $\dots$};
		
		\draw (G-0) -- (G-1) node [midway, left=2pt] {$\scriptstyle c(a_1)= 3$};
		\draw (G-0) -- (G-2) node [midway, below right=1pt] {$\scriptstyle c(d_1)= 4$};
		\draw (G-1) -- (G-2) node [midway, right=3pt] {$\scriptstyle c(b_1)= 2$};
		
		\draw (G-0) -- (G-3) node [midway, right=10pt] {$\scriptstyle c(a_2)= 30$};
		\draw (G-0) -- (G-4) node [midway, left=3pt] {$\scriptstyle c(d_2)= 40$};
		\draw (G-3) -- (G-4) node [midway, below right=1pt] {$\scriptstyle c(b_2)= 20$};
		
		\draw (G-0) -- (G-n-1) node [midway, below left = 0.1pt] {$\scriptstyle c(a_{\frac{n}{2}})= 3 \cdot 10^{\frac{n-2}{2}}$};
		\draw (G-0) -- (G-n) node [midway, above = 9pt] {$\scriptstyle \qquad \quad c(d_{\frac{n}{2}})= 4 \cdot 10^{\frac{n-2}{2}}$};
		\draw (G-n-1) -- (G-n) node [midway, left=2pt] {$\scriptstyle c(b_{\frac{n}{2}})= 2 \cdot  10^{\frac{n-2}{2}}$};
		
	\end{tikzpicture}

	\caption{Intractability example: Graph (see Proof of Theorem~\ref{Theorem_Intractable}) \label{exm_intractable_nocosts}}
	\label{exm_intractable}
\end{figure}

\begin{corollary}
	The Generalized Cable-Trench Problem (\ref{BGCTP}) is intractable. 
	\label{Theorem_GCTPIntractable}
\end{corollary}

\section{Integer Programming Formulation} \label{sec:IPFormulation}
Since the (generalized) cable-trench problem, just like any other combinatorial program, can be formulated as an integer program and solved as such, we present a bi-objective integer linear programming formulation of the CTP and GCTP.
\setlength{\jot}{7pt}
\begin{subequations}
	\begin{align}
		\text {(CTP)} 
		&& \min\; & \sum_{i=1}^{n} \sum_{i=1}^{n} c_{ij} \, x_{ij} \\
		&& \min\;& \sum_{i=1}^{n} \sum_{j=i+1}^{n} c_{ij} \, y_{ij} \\
		&& \text{s.t.}\; & \sum_{j=1}^{n} x_{1j} = n-1 && \label{shortestpatheq1} \\
		&& & \sum_{j=1}^{n} x_{ij} - \sum_{k=1}^{n} x_{ki}  =  -1 && \text{for all } i \in \{2, 3, \dots, n\} \label{shortestpatheq2} \\
		&& & \sum_{i=1}^{n} \sum_{j=i+1}^{n} y_{ij}  =  n-1 \label{msteq} \\
		&& & (n-1) \, y_{ij} - x_{ij} - x_{ji} \geqslant 0 && \text{for all } i,j \in \{1, \dots, n\} \text{ with } i<j \label{couplinconstraint} \\
		& & & x_{i1} = 0 && \text{for all } i \in \{1, \dots, n\} \\
		&& & x_{ij}  \in \Z_{\geqslant} && \text{for all } i,j \in \{1, \dots, n\} \text{ with } j \neq 1 \label{shortestpathvarconstraint} \\
		&& & y_{ij}  \in  \{0,1\} && \text{for all } i,j \in \{1, \dots, n\} \text{ with } i<j \label{mstvarconstraint}
	\end{align}
\end{subequations}

Variables $x_{ij}$ represent the number of cables from vertex $i$ to vertex $j$, $y_{ij}$ is $1$ if a trench is dug between vertex $i$ and $j$ ($i < j$), $0$ otherwise, and $c_{ij}$ represents the respective costs for choosing the edge that connects vertex $i$ and vertex $j$. Hence, variables $x_{ij}$ can be interpreted as flow variables modelling the shortest path. Variables $y_{ij}$ on the other hand can be interpreted as those representing the spanning tree. Note that no cables flow back to vertex $1$. If no such edge $e = (i,j)$ between vertex $i$ and $j$ exists, we set $x_{ij} = x_{ji} = y_{ij} = 0$.

Now, we will have a closer look at the set of the constraints. Constraint \eqref{couplinconstraint} can be considered as a coupling constraint of  the two variable types and thus of the two objective functions. This restriction ensures, that (shortest) paths can only use edges which are part of the spanning tree solution. Constraint \eqref{shortestpatheq1} and \eqref{shortestpatheq2} refer to flow conservation constraints modelling the single-source shortest path problem. Constraint \eqref{msteq} is part of the minimum spanning tree problem. It ensures that exactly $n-1$ edges are chosen. In other words, the edges form a spanning tree. This restriction can be omitted without changing the optimal solution, but it considerably strengthens the formulation of the CTP. We will discuss this later in more detail. Note that subtour elimination constraints are not necessary in the formulation of the cable-trench problem as constraints \eqref{shortestpatheq1},  \eqref{shortestpatheq2} and \eqref{couplinconstraint} ensure that in any efficient solution $(\bar{x}, \bar{y})$, the graph is connected and the edge set $\{ [i,j] \in E : y_{ij} = 1\}$ is cycle free and thus a spanning tree. 
In fact, the problem can be relaxed by omitting integrality constraint \eqref{shortestpathvarconstraint}. This can be easily seen from the following. Given $y_{ij} \in \{0,1\}$ fixed for all edges $(i,j) \in E$. The resulting problem that needs to be solved ends up being a single-source shortest path problem on a subgraph, which has an integral polyhedral formulation. A proof and further reading can be found in \cite{Wols1998}. However, the integrality constraints of the minimum spanning tree variables can not be omitted. Neither constraint \eqref{couplinconstraint} nor constraint \eqref{msteq} force $y_{ij}$ to be integral. Most likely the contrary is true. Due to the minimization of the second objective $y_{ij}$ is chosen just large enough to satisfy the bigM constraint \eqref{couplinconstraint}. Accordingly, the CTP can be formulated as a bi-objective zero-one mixed-integer linear program.

Similarly, the GCTP can be formulated by simply substituting the two objective functions of the CTP as follows: 
\begin{subequations}
	\begin{align}
		\text {(GCTP)} 
		&& \min\; & \sum_{i=1}^{n} \sum_{i=1}^{n} s_{ij}\, x_{ij}  && \\
		&& \min\; & \sum_{i=1}^{n} \sum_{j=i+1}^{n} t_{ij}\, y_{ij} &&  \\
		&& \text{s.t.}\; & \eqref{shortestpatheq1} - \eqref{mstvarconstraint} \notag
	\end{align}
\end{subequations}

\section{Multi-objective Solution Strategy}  \label{sec:SolutionStrategy}
Scalarization is the most frequently applied solution approach not only for MOCO, but for multi-objective optimization in general. It is based on the idea of combining objective functions and/or reformulating objectives as constraints to obtain a single-objective optimization problem. In the following, we will discuss those techniques that are most closely associated with our work and the problem introduced by \cite{Vasko} itself. For a detailed literature review of solution methods for MOCO we refer to \cite{EhrGand2000}. For an introduction to scalarization methods we refer to \cite{Ehr2005}.

One of the most common scalarization techniques is the \emph{weighted sum scalarization} which has already been discussed in Section \ref{Section Bi-Objective CTP}. The intuitiveness of this technique not only makes it easy for a decision maker to understand, but also provides a comparably fast method to find all (extreme) supported non-dominated points (see, e.g., \cite{AnejaNair} and \cite{Cohon}). However, the weighted sum approach generally fails in computing the entire Pareto set, since it is not possible to find unsupported solutions, as we have already seen for the (generalized) cable-trench problem as well. One way to overcome that is to use the two phase method \citep{UlungTeg1995}. In the first phase the set of supported non-dominated points is computed, for example, via the weighted sum approach. In the second phase the set of non-supported non-dominated points is determined in any box defined by adjacent extreme supported non-dominated points and their local nadir point. Tighter bounds can improve the procedure as discussed, for example, in \cite{Przy2011}.

Another frequently applied scalarization technique is the \emph{$\epsilon$-constraint method}, which in contrast to the weighted sum is able to generate supported and non-supported non-dominated points. The method was first introduced by \cite{Haimes1971} and further discussed in \cite{Chankong1983}. Instead of aggregating both objectives of the BOCO, only one of the objectives is selected to be optimized while the other one is incorporated as an additional constraint. The resulting $\epsilon$-constraint problem to be solved is given by
\begin{subequations}
	\begin{alignat}{3}
		(\text{P}_{\epsilon})
		&\qquad \qquad \qquad &\min \;& f_i(x)  &&	
		\\
		&&\text{s.\,t.} \;& f_k(x) \leqslant \epsilon &\qquad k = 1,2 \quad k \neq i \label{addepsconstraint}\\
		&&& x \in X
	\end{alignat}
\end{subequations}
where $\epsilon \in \R$. By substituting the objective of $\text{P}_{\epsilon}$ by $\{ \min f_i(x) + \upalpha f_k(x), k \neq i\}$ with $\upalpha$ sufficiently small, also known as the hybrid method \citep{Ehr2005}, it is possible to determine the set $Y_N$ without computing weakly efficient solutions as well.

In this work, we use the hybrid $\epsilon$-constraint approach to solve the B(G)CTP and further improve it by a problem-specific cutting plane. In particular, we use as objective function of the $\epsilon$-constraint problem the objective of the shortest path problem and incorporate the minimum spanning tree objective as additional constraint and as augmentation term of the hybrid method. The resulting $\epsilon$-constraint problems for both, BCTP and BGCTP, are given by:
\begin{subequations}
	\begin{alignat}{6}
		(\text{P}_{\epsilon})
		& \quad  &\min \;& c_{\gamma}(T) + \alpha \,c_{\tau}(T) & \qquad \qquad \qquad (\text{P}_{\epsilon}^{G}) & \quad  &\min \; & s_{\gamma}(T) + \alpha\, t_{\tau}(T) \notag \\
		&&\text{s.\,t.} \;& c_{\tau}(T) \leqslant \epsilon & && \text{s.\,t.} \;& t_{\tau}(T) \leqslant \epsilon \notag \\
		&&& T \in \mathcal{T} &&&& T \in \mathcal{T} \notag
	\end{alignat}
\end{subequations}

To compute the whole non-dominated set for the CTP, we iteratively solve $\epsilon$-constraint problems $\text{P}_{\epsilon^{r}}$, for which we set $\epsilon^{r} :\,= c_{\tau}(T^{r-1}) - 1$ in iteration $r$. In the first iteration $r=1$, we set $T^0 = \text{lexmin}_{T \in \mathcal{T}}(c_{\gamma}, c_{\tau})$. Otherwise, $T^{r-1}$ is the optimal solution of $\text{P}_{\epsilon^{r-1}}$ determined in the last iteration $r-1$. An analogous approach can be followed for the GCTP. Moreover, in order to avoid computing weakly efficient solutions, we set $\upalpha = 1 / d_{lex}$, where $d_{lex}$ denotes the precomputed largest difference in terms of the minimum spanning tree objective that can occur between two efficient solutions. Thus, we can guarantee to find the whole set $Y_N$.

Although the $\epsilon$-constraint method offers a way of completely solving the (G)CTP, we are still faced with the difficulty of its underlying complexity in terms of hardness and intractability, that we want to address here. To obtain one non-dominated point, one $\epsilon$-constraint problem needs to be solved, that is, a total of $\abs{\YN}-1$ times. Note that the lexicographic minima are precomputed.

\subsection{Cutting Plane}

Regardless of the chosen scalarization method, we still face the task of solving a sequence of single-objective integer programs (IPs). For the $\epsilon$-constraint approach, we see this immediately. In order to speed up the solution process for these IPs, the inclusion of problem specific cutting planes may be useful. The general idea is to add further constraints to the system to reduce the solution space, but without cutting off optimal integral solutions, see, e.g., \cite{Wols1998}. For example, constraint \eqref{msteq} can be understood as such a cutting plane. The cutting plane that we introduce in the following is specifically designed for the $\epsilon$-constraint problem.

Recall that in iteration $r+1$ the value $\epsilon$ is set to $\epsilon^{r+1} :\,= c_{\tau}(T^{r}) - 1$, the total edge cost of the optimal solution \((\bar{x}^r, \bar{y}^r)\) determined in iteration \(r\) minus one such that the solution \(T^r\) of the previous iteration becomes infeasible. Since exactly \(n-1\)  components of  \(\bar{y}^r\) are equal to one (to form a spanning tree), the minimum spanning tree must have an objective function value that is greater or equal to the \(n-1\) cheapest edges.
Consequently, any feasible solution $(x,y)$ in iteration $r$ satisfies
\begin{align*}
	\max_{\substack{i=1 \dots n \\ j = i+1 \dots n }} \left\{ c_{ij} y_{ij} \right\}  \leqslant \epsilon - \sum_{e \in S} c_e \quad \Leftrightarrow \quad c_{ij} y_{ij} \leqslant \epsilon - \sum_{e \in S} c_e \quad i,j \in \{1, \dots, n\} :  i < j,
\end{align*}
where $S$ be the set of the $n-2$ cheapest edges.

Note that the cut is aplicable for the GCTP as well by simply considering as edge costs for this cut the respective edge cost coefficients of the minimum spanning tree problem. Later, we refer to the cut as $\epsilon$-cut.

\section{Numerical Results}\label{sec:NumericalResults}
In order to evaluate the chosen solution strategy and particularly the benefit of the presented cutting plane, we examine numerical results on several test problem classes, that differ in their respective graph structure. In the following, we introduce those test sets and how they have been generated. After that we present the framework of our implementation and discuss resulting outcomes.

\subsection{Test Problems}

The considered test instances are classified according to their number of arcs and their underlying structure. All graphs are connected and each edge in a given graph is assigned two costs, where the range of  costs is between 1 and 100. 

The first test problem classes that we introduce vary primarily in their density. We generated \emph{incomplete} graph instances for which their densities are given by $d_1 = 0.125,\, d_2 = 0.25,\, d_3 = 0.5$ and $d_4 = 0.75$. These graphs are generated in the following way: to a spanning tree (to ensure the connectivity of the graph) additional edges are added until the desired density is satisfied. All edges are randomly added to the graph via a pseudo random number generator with a linear congruential generator algorithm, proposed by NETGEN \citep{netgen}. The edge costs are generated via the same pseudo random number generator. We test our solution approach also on \emph{complete} graphs $K_n$ with a density $d_5 = 1$, for which each pair of distinct vertices is connected by an undirected edge. The random number generator for the costs remains the same as for the incomplete graphs. 

Along with test problem classes that differ primarily in their density, we also consider classes that show a characteristic graph structure, such as \emph{grid} and \emph{location based} graphs. For both, the vertices of a graph lie in the same plane. To generate location based graphs, the number of vertices and arcs can be set manually. The vertices are determined randomly in the plane, following a) the standard normal distribution and b) the uniform distribution proposed by MATLAB. For the CTP, edge costs are set either by the Euclidean or the Manhattan distance. 
For the GCTP, the edges are assigned both the Euclidean and the Manhattan distance as costs. For both cases, a) and b), three ways of choosing arcs to build an instance are distinguished: randomly chosen arcs, arcs with minimum distance (w.r.t.\ either Euclidean or Manhattan distances). All graphs are created such that connectivity is ensured and the number of edges is equal to the predefined parameter selection. For our test setup, the density for each of the six test sets is $0.5$. In contrast, grid graphs are characterized by a two-dimensional rectangular grid constructed by the vertices. We generated undirected graph instances that differ only in their costs. The grid width is chosen as the rounded square root of the number of vertices. The random generator for the costs remains that of NETGEN. 

In total, we distinguish four test problem classes (\emph{incomplete, complete, location based, grid graph}) and 12 test sets (\emph{incomplete} graphs with a density of \emph{0.125, 0.25, 0.5, 0.75}, \emph{complete} graphs, \emph{location based} graphs with \emph{uniformly} and \emph{standard normally} distributed points for both of which arcs are chosen \emph{randomly}, with minimum \emph{Euclidean} and \emph{Manhattan} distance, \emph{grid} graphs). For all 12~test sets, we generated several sets of instances according to the number of vertices. Thus, each instance set is classified due to its test problem class and its test set, as well as by its number of vertices. For each instance set, 20 different graph instances per number of vertices are created. 

We test different problem sizes for the CTP and the GCTP, respectively. For the CTP, we examine instances with a number of $20, 25, \ldots, 75,100$ vertices. However, most of the test instances with 100 vertices could not be completely solved within the predefined time limit of 300 seconds. For the GCTP, on the other hand, we are testing instances with a number of $15, 20, \ldots, 45, 50$ vertices.

\subsection{Test Results}

The computational results we present in this section can be seen as an attempt to address several questions about the overall performance in solving the bi-objective (generalized) cable-trench problem with the chosen solution strategy.

\begin{enumerate}
	\item How well does the $\epsilon$-constraint method perfom on the bi-objective (G)CTP? Up to which problem size can test instances be solved? Are there significant differences in solving the CTP and the GCTP?
	\item Are there any distinctive features regarding the level of computational effort between the different test problem classes or test sets? How important is the graph density, cost and graph structure?
	\item What is the effect on tightening the formulation of the $\epsilon$-constraint method by adding an extra cut? Does an effect depend on a specific graph characteristic? 
\end{enumerate}

We start our evaluation by looking at the test results of the test sets that differ only in their density and have no other special graph structure, i.e., the \emph{incomplete} and \emph{complete} graphs. 

\begin{figure}
	
	\begin{minipage}{0.48\textwidth}
		\includegraphics[width=1\textwidth]{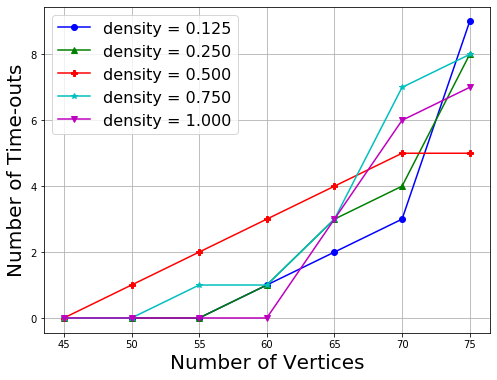}
		\caption{CTP Time-outs, (in)complete graph}
		\label{ctptimeouts_Fig}
	\end{minipage}
	\hfill
	\begin{minipage}{0.48\textwidth}
		\includegraphics[width=1\textwidth]{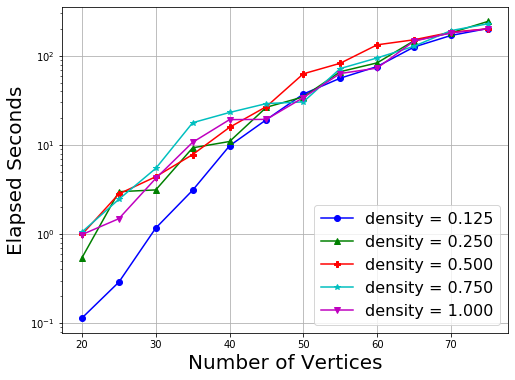}		
		\caption{CTP Runtime, (in)complete graph}
		\label{ctpruntime_Fig}
	\end{minipage}\\[0.5cm]
	\begin{minipage}{0.48\textwidth}
		\includegraphics[width=1\textwidth]{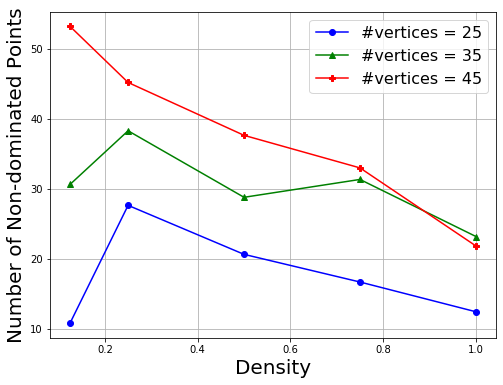}
		\caption{CTP \#non-dominated points, (in)complete graph}
		\label{dens_ndpoints_Fig}
	\end{minipage}
	\hfill
	\begin{minipage}{0.48\textwidth}
		\includegraphics[width=1\textwidth]{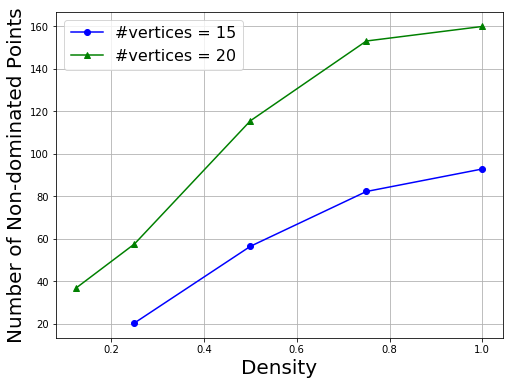}
		\caption{GCTP \#non-dominated points, (in)complete graph}
		\label{dens_ndpoints_gctp_Fig}
	\end{minipage}\\[0.5cm]
	\begin{minipage}{0.48\textwidth}
		\includegraphics[width=1\textwidth]{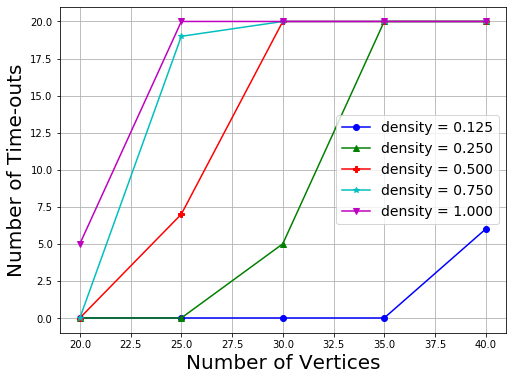}
		\caption{GCTP Time-outs, (in)complete graph}
		\label{gctptimeouts_Fig}
	\end{minipage}	
	\hfill
	\begin{minipage}{0.48\textwidth}
		\includegraphics[width=1\textwidth]{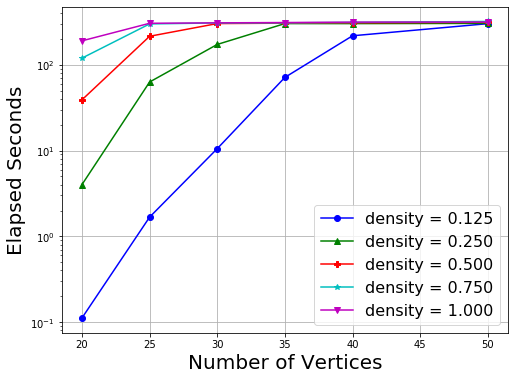}
		\caption{GCTP Runtime, (in)complete graph}
		\label{gctpruntime_Fig}
	\end{minipage}
\end{figure}

{ \small
	\begin{table}
		\begin{center}
			\begin{tabular}[h]{l c c c c c }\toprule
				&  \multicolumn{5}{c}{t [s] / \(|Y_N|\)} \\ 
				\midrule
				& $d = 0.125$ & $d = 0.25$ & $d = 0.5$ & $d = 0.75$ & $d = 1$  \\
				\midrule
				30 & 0.051 & 0.112 & 0.167 & 0.229 & 0.251 \\ 
				50 &  0.640 &  0.627 &  1.260 & 0.958 & 1.219 \\ 
				70 & 1.693 & 2.455 & 3.737 & 4.449 & 5.708 \\
				\bottomrule
			\end{tabular}
		\end{center}
		\caption{Elapsed seconds per non-dominated point}
		\label{s_ndpoints_Tab}
	\end{table}
}

In Figure~\ref{gctptimeouts_Fig}, we see that all instances of the GCTP with up to 35 vertices in a graph instance with a density $d = 0.125$ can be solved within the predefined time window. Test instances with a higher number of nodes can no longer be solved completely. As the density increases, the number of nodes for which instances can still be solved completely decreases. For the CTP, we see in Figure 4, that instances with up to 45 can be completely solved within the maximum time window, regardless of the density of the graph. Thus, we observe that solving the CTP is less computationally expensive than solving its generalized version. This was to be expected, as the cost coefficients for both objectives of the GCTP are generally different, thus increasing the potential for conflict between the objectives. With this, the first question we wanted to answer has been addressed.

Results that we can take from the first observations lead us to the direct discussion of the second question, namely whether there are differences in solving the (G)CTP with respect to the different test problem classes. As already stated, the computational effort needed to solve the GCTP increases with the density of the underlying graph of the problem. This can be seen in Figure~\ref{gctpruntime_Fig} as well. Again, this result is to be expected. With an increasing density, the number of variables increases and thus the set of feasible solutions as well. A higher number of feasible solutions merely leads to more feasible solutions of the two associated single-objective problems, that are generally in conflict, since the cost coefficients for both objective functions are independent of each other. The fact that the different cost coefficients potentially lead to more conflicts as the number of variables increases can be observed particularly from the results of the CTP, for which we obtain significantly different results with regard to the density of the underlying graph. In Figure~\ref{ctptimeouts_Fig} and in Figure~\ref{ctpruntime_Fig}, we see that with an increasing density the runtime to solve an instance does not necessarily increases as well. In fact, we even find that the runtime for solving complete graphs actually is less in comparison to instances with a lower density. From Table~\ref{s_ndpoints_Tab}, we can conclude that those runtime savings are not due to the fact that the time to determine one non-dominated point, i.e., to solve one $\epsilon$-constraint problem, decreases. On the contrary, the time to determine one non-dominated point increases with the density of the graph. This is to be expected, since the number of edges and thus the size of the resulting $\epsilon$-constraint problem to be solved increases. However, in Figure~\ref{dens_ndpoints_Fig}, we see that the number of non-dominated points is decreasing while the density is increasing. From that, we can follow that the two objective functions are less conflicting as the number of edges increases, which in turn explains that the runtime is decreasing with increasing density. This is in contrast to the GCTP, for which the number of non-dominated points is increasing as the density increases, see Figure~\ref{dens_ndpoints_gctp_Fig}. These different effects may be explained by the cost coefficients. Ultimately, the conflict between the objectives of the CTP is due to the fact that the cost coefficient of an edge $e$ with $y_e = 1$ is counted once in the minimum spanning tree objective while it is counted $k$ times if edge $e$ is occurring in $k$ paths. The conflict exists above all, when the chosen edges to construct a low-cost spanning tree lead to long and thus expensive paths in the shortest path objective. Then, the two objectives are in great conflict. Conversely, if the graph is rather dense the shortest paths from the root vertex to the other vertices often consist of a small number of edges. In this case, the two objectives are less conflicting, since the selected edges imply lower costs for both objectives due to the same low cost coefficients. However, the same is not true for the GCTP, for which the cost coefficients differ for both objectives. That is why, the effect we observe for the CTP cannot be observed for the GCTP.

{ \small
	\begin{table}
		\begin{center}
			\begin{tabular}[h]{p{2.2cm}lc c  c c }\toprule
				& & \multicolumn{2}{c}{CTP} & \multicolumn{2}{c}{GCTP }\\ 
				\midrule
				distribution of\newline points in the plane & \parbox{2.2cm}{\vspace{0.25cm}selected arcs \newline($d = 0.5$)} & \(|Y_N|\) & t [s] & \(|Y_N|\) & t [s]  \\
				\midrule
				& random & 70.00 & 39.69  & 71.85 & 37.75 \\ \cmidrule{2-6}
				uniformly &  \parbox{2.4cm}{\vspace{0.1cm}min Euclidean\vspace{0.1cm}} & 66.55 & 114.77 & 76.00  & 116.99 \\ \cmidrule{2-6}
				& \parbox{2.4cm}{\vspace{0.1cm}min Manhattan\vspace{0.1cm}}   & 60.65 & 78.14 & 61.45 & 74.95 \\ \midrule
				& random & 52.95 & 33.58 & 58.70 & 36.60 \\\cmidrule{2-6}
				normally & \parbox{2.4cm}{\vspace{0.1cm}min Euclidean\vspace{0.1cm}}  & 45.65 & 58.00 & 50.10 & 74.81 \\ \cmidrule{2-6}
				& \parbox{2.4cm}{\vspace{0.1cm}min Manhattan\vspace{0.1cm}} & 39.00 & 31.07 & 40.15 & 32.34 \\ 
				\bottomrule
			\end{tabular}
		\end{center}
		\caption{Number of non-dominated points for different problem sizes}
		\label{locationbased_Tab}
	\end{table}
}
Now, we want to analyze, whether the distribution of the cost coefficients may have an impact on solving the (G)CTP. Therefore, we examine the results of the test instances of \emph{location based} graphs, that allow us to create graph sets with different cost structures. In particular, we test instances with a density $d = 0.5$, for which we select edges according to different strategies. The results of solving instances with 25 vertices can be found in Table \ref{locationbased_Tab}. On the one hand, the distribution of points in the plane influences the length (Euclidean and Manhattan) of an edge and thus the cost structure of the different test sets. On the other hand, the connectivity and cost structure of instances differ due to the selection of the edges. The cost coefficients in the graph instances with a minimum Euclidean or minimum Manhattan distance edge selection vary less than randomly selected edges. Moreover, in test instances in which points are normally distributed, the edge costs show less variation than for uniformly distributed points. In Table \ref{locationbased_Tab}, we see that these differences in the cost structure have an impact on the number of non-dominated points as well as the runtime. The first observation from Table~\ref{locationbased_Tab} is that we get similar results in solving the CTP and the GCTP. Regardless of the edge generation strategy, the cost coefficients of the two objective functions of the GCTP are in our test instances positively correlated. Recall, that the cost coefficients of the first objective function correspond to Euclidean distances and of the second objective function to Manhattan distances.  Therefore, properties of the CTP are preserved when solving the GCTP. Another observation is that all test sets for which the points are normally distributed in the plane have a lower number of non-dominated points as well as less computational runtime in comparison to the test sets for which the points are uniformly distributed. In general, this may be explained in a similar way as for dense or complete graphs, since it can be assumed that a generated graph from normally distributed points in the plane is relatively dense in a center with few edges outside of the center. Note that the same effects can be observed by testing instances with a different number of vertices.

To conclude the second research question, we investigate the effect of special graph structures. Therefore we compare the results of the generated grid graphs with sparse graphs of the same density. Figure~\ref{runtime_grid_Fig}, \ref{timeouts_grid_Fig} (for the CTP) and Figure~\ref{runtime_grid_gctp_Fig}, \ref{timeouts_grid_gctp_Fig} (for the GCTP) show that the computational effort to solve instances of grid graphs is higher in comparison to instances with similar density but no specific structure. This effect was to be expected due to the highly structured construction of grid graphs as such. The highly symmetrie structure of grid graphs regarding the node degree and the local connection structure lead to a lot of possible paths from one vertex to another. Thus, there is even for small instances a high potential of conflicts between the two objectives as there is no edge that needs to be contained in a solution to guarantee a spanning tree.

\begin{figure}
	\begin{minipage}{0.48\textwidth}
		\includegraphics[width=1\textwidth]{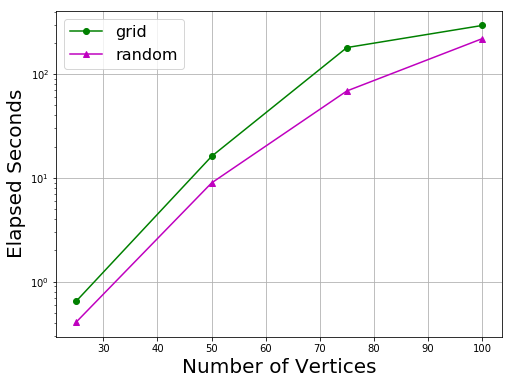}
		\caption{CTP Runtime, gridgraph}
		\label{runtime_grid_Fig}
	\end{minipage}
	\hfill
	\begin{minipage}{0.48\textwidth}
		\includegraphics[width=1\textwidth]{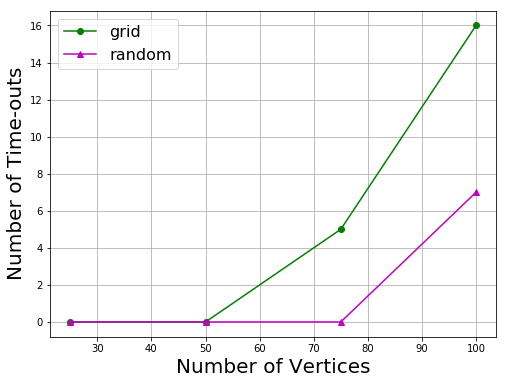}
		\caption{CTP Time-outs, gridgraphs}
		\label{timeouts_grid_Fig}
	\end{minipage}\\[0.5cm]
	\begin{minipage}{0.48\textwidth}
		\includegraphics[width=1\textwidth]{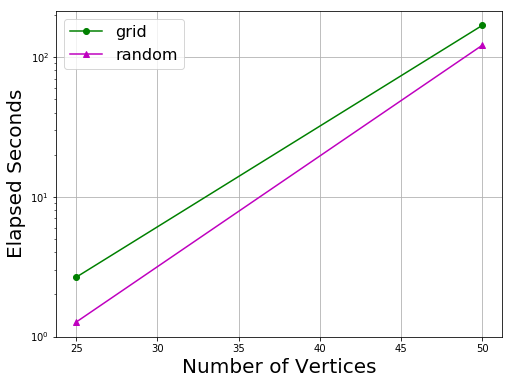}
		\caption{GCTP Runtime, gridgraph}
		\label{runtime_grid_gctp_Fig}
	\end{minipage}
	\hfill
	\begin{minipage}{0.48\textwidth}
		\includegraphics[width=1\textwidth]{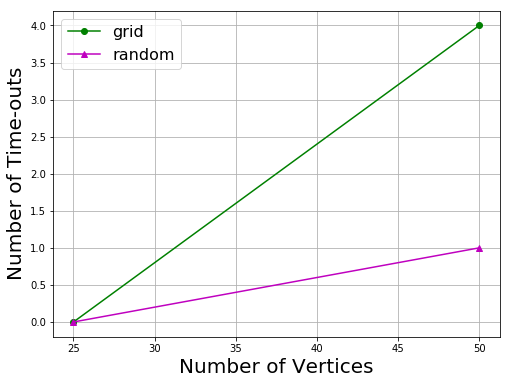}
		\caption{GCTP Time-outs, gridgraph}
		\label{timeouts_grid_gctp_Fig}
	\end{minipage}
\end{figure}

\begin{figure}
	\begin{minipage}{0.48\textwidth}
		\includegraphics[width=1\textwidth]{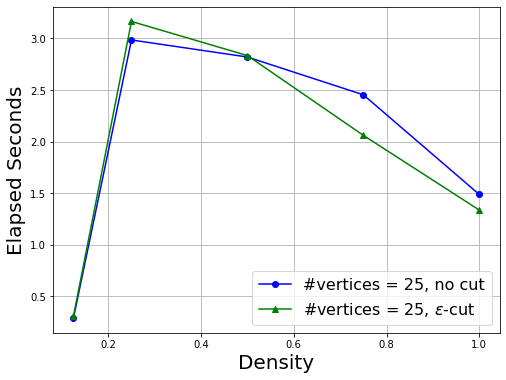}
		\caption{CTP Runtime, 25 vertices}
		\label{runtime_epscut25_Fig}
	\end{minipage}
	\hfill
	\begin{minipage}{0.48\textwidth}
		\includegraphics[width=1\textwidth]{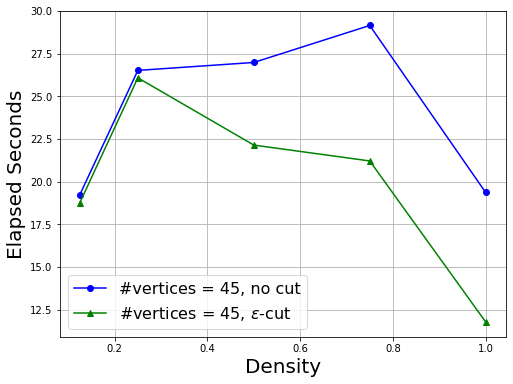}
		\caption{CTP Runtime, 45 vertices}
		\label{runtime_epscut45_Fig}
	\end{minipage}\\[0.5cm]
	\begin{minipage}{0.48\textwidth}
		\includegraphics[width=1\textwidth]{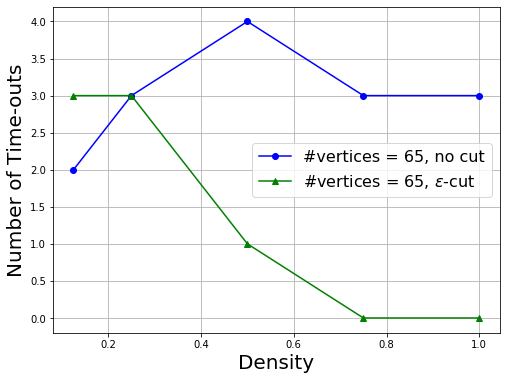}
		\caption{CTP Time-outs, 65 vertices}
		\label{timeouts_epscut65_Fig}
	\end{minipage}
	\hfill
	\begin{minipage}{0.48\textwidth}
		\includegraphics[width=1\textwidth]{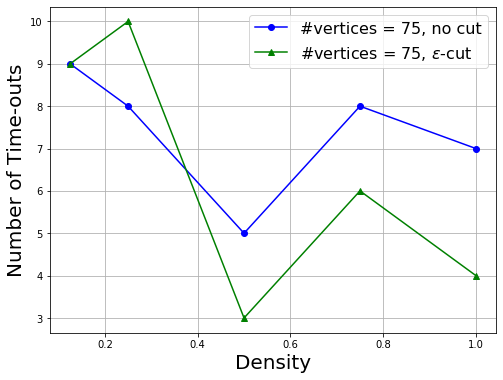}
		\caption{CTP Time-outs, 75 vertices}
		\label{timeouts_epscut75_Fig}
	\end{minipage}	
	
\end{figure}

Finally, we examine the third question regarding the effect of the introduced $\epsilon$-cut. We applied the cut to the test problem classes of \emph{incomplete} and \emph{complete} graphs. The runtime as well as the number of time-outs are compared to the test results, for which no cut was added. In Figure~\ref{runtime_epscut25_Fig} and \ref{runtime_epscut45_Fig}, we see that as the density of the graph instances increases, the runtime of the program that uses the $\epsilon$-cut is lower than the runtime of the standard $\epsilon$-constraint method. If we look at the results for instances that have exceeded the maximum time window, we observe the same effect, see Figure~\ref{timeouts_epscut65_Fig} and \ref{timeouts_epscut75_Fig}. With increasing density, the number of total time-outs decreases comparatively by adding the $\epsilon$-cut. Thus, the higher the density of the underlying graph for which the CTP has to be solved, the more effective the additional cut is. Due to the definition of the $\epsilon$-cut itself, this effect is plausible. The cut affects particularly the edges of the underlying graph instance, meaning edges are possibly cut off from the decision space. Since for dense graphs the paths in the tree are rather short w.r.t.\ the number of edges, there is less conflict between the shortest path and the spanning tree objective. 
However, the size of the graph in terms of the number of edges causes a substantial part of the runtime, see Table~\ref{s_ndpoints_Tab}, and thus, cutting of edges may have a significant impact on the runtime. Moreover, in dense graphs, there are potentially more edges to cut off. The effectiveness in relation to the density of the graph can be observed regardless of its number of nodes. However, there is no such effect regarding the GCTP, since for the GCTP the objectives are conflicting already due to the different cost coefficients of the objectives. In comparison to the CTP the conflict between the two objective functions is not decreasing as the density of the graphs is increasing. Therefore, the $\epsilon$-cut is not as effective as for the CTP. 

In the following, we outline the framework of our implementation and give a short description of the computational resources of the machine on which the tests are performed. 
The $\epsilon$-constraint method that we use to solve the (generalized) cable-trench problem is implemented in C++. Each single-objective mixed-integer $\epsilon$-constraint problem that has to be computed in one step of the algorithm is solved with CPLEX 20.1.0. All numerical tolerances correspond to the default values specified by CPLEX. To measure the performance of the implemented algorithm and cutting plane, the average number of non-dominated points, the average runtime in seconds, the average nodes explored during the branch-and-bound procedure and the number of time-outs are stored. For each graph instance, we set a time limit of 300 seconds to determine the complete non-dominated set. If a problem instance can not be solved within the predefined time window, the current number of determined non-dominated points is stored and the number of time-outs within the instance set is  increased by one. 

All tests are performed on a machine with a Intel\textregistered\ Core\texttrademark\ i7-8700 CPU 3.20\,GHz with 12\,MB cache and 32\,GiB RAM. The implementation and the used test instances can be found in \cite{loehken23bgctp}.

\section{Concluding Remarks}\label{sec:Conclusion}
In this article we show that the biobjective formulation of the classical cable-trench problem can lead to new insights on the structure and complexity of CTP. Moreover, with the biobjective formulation unrevealed (efficient) compromise solutions can be determined. 

The numerical test show that already moderate instance sizes pose a considerable challenge for the  optimization of the \(\epsilon\)-constraint scalarization using CPLEX. The numerical difficulty is due to the intractability of the biobjective optimization problem and the \NP-hardness of 
its \(\epsilon\)-constraint scalarization. Interestingly, these two effects can be observed seperately in the numerical results. Dense graphs tend to have a small number of nondominated solutions for which each \(\epsilon\)-constraint scalarization is hard to solve due to the large number of variables. Sparse graphs however, have rather small numbers of variables, but due to the longer paths more conflicts between the objectives which increases the number of nondominated points. 

Even though CTP and the generalized CTP (GCTP) look very similar, the objective functions in their biobjective formulations are conflicting due to different reasons. While the objective coefficients in GCTP are different and can in general be uncorrelated (or negatively correlated) which yields conflicting objective, in CTP the objective function coefficients are the same in both objectives. 
The objective functions thus only only differ in the way the edge costs are aggregated. While the spanning tree objective accounts for every edge in the tree, the shortest path objective counts the cost of an edge in the tree as many times as this edge is on a shortest path, which leads to a conflict of objective in particular if the number of edges in the shortest paths increases.

\paragraph{Acknowledgments}
The authors thank Kathrin Klamroth for the discussions which helped to elaborated the intractability proof. Futhermore, the authors thankfully acknowledge financial support by Deutsche For\-schungs\-gemeinschaft, project number KL~1076/11-1.

\bibliographystyle{abbrvnat}
\bibliography{literatur.bib}

\end{document}